\documentclass[12pt]{amsart}
\usepackage{amsmath,amssymb,enumerate,mathtools,dsfont,bbding}
\newtheorem{theorem}{Theorem}[section]
\newtheorem{claim}{}[theorem]
\newtheorem{lemma}[theorem]{Lemma}

\newtheorem{conjecture}[theorem]{Conjecture}
\theoremstyle{definition}

\newcommand{\bF}{\mathbb F}

\newcommand{\bN}{\mathbb N}

\newcommand{\eps}{\varepsilon}

\DeclareMathOperator{\cl}{cl}

\DeclareMathOperator{\PG}{PG}

\DeclareMathOperator{\GF}{GF}
\DeclareMathOperator{\AG}{AG}

\newcommand{\del}{ \backslash  }

\numberwithin{subcase}{case}
\numberwithin{subsubcase}{subcase}
\newenvironment{subproof}[1][\proofname]{%
  \begin{proof}[Subproof:]%
}{%
  \end{proof}%
}

\usepackage[margin=1.5in]{geometry}
\begin{document}
%\sloppy

\title{The Critical Number of $I_{1,t}$-free triangle-free binary matroids}
\author[Nelson]{Peter Nelson}
\address{Department of Combinatorics and Optimization, University of Waterloo, Waterloo, Canada. Email address: {\tt 	apnelson@uwaterloo.ca}}
\author[Nomoto]{Kazuhiro Nomoto}
\address{Department of Combinatorics and Optimization, University of Waterloo, Waterloo, Canada. Email address: {\tt 	knomoto@uwaterloo.ca}}
\thanks{This work was supported by a discovery grant from the Natural Sciences and Engineering Research Council of Canada and an Early Researcher Award from the government of Ontario}

\date{\today}
\begin{abstract}
	A simple binary matroid, viewed as a restriction of a finite binary projective geometry $\PG(n-1,2)$, is $I_{1,t}$-free if for any rank-$t$ flat of $\PG(n-1,2)$, its intersection with the matroid is not a one-element set.
	In this paper, we show that the simple $I_{1,t}$-free and triangle-free binary matroids have bounded critical number for any $t \geq 1$. 
\end{abstract}

\maketitle

\begin{section}{Introduction}
In this paper, we show that the simple $I_{1,t}$-free triangle-free binary matroids have bounded critical number for any $t \geq 1$. A simple binary matroid, viewed as a restriction of a finite binary projective geometry $\PG(n-1,2)$, is \emph{$I_{1,t}$-free} if its intersection with any rank-$t$ flat of $\PG(n-1,2)$ is not an one-element set. The critical number is the quantity $n-k$ where $k$ is the dimension of a largest subgeometry of $\PG(n-1,2)$ disjoint from the matroid. 

To motivate our result, we consider the following conjecture of Bonamy, Kardo\v{s}, Kelly, Nelson and Postle [\ref{bkknp}]. A simple binary matroid is \emph{$I_{s}$-free} if it has no $s$-element independent flat.

\begin{conjecture}[{[\ref{bkknp}]}]\label{chi_bounded_conj}
For any $s \geq 1$, the simple $I_s$-free and triangle-free binary matroids have bounded critical number.
\end{conjecture}

This conjecture resembles the famous Gy{\'a}rf{\'a}s-Sumner Conjecture for graphs ([\ref{g85}],[\ref{s81}]), which claims that, for any tree $T$ and clique $K$, the graphs that omit $T$ and $K$ as induced restrictions have bounded chromatic number. It is similar in the sense that omitting a tree, which is maximally acyclic, is analogous to $I_s$-freeness. Conjecture \ref{chi_bounded_conj} corresponds to the special case where $K$ is a triangle. It is worth noting that the conjecture in fact fails if one replaces `triangle-free' with `$F_7$-free'. Bonamy, Kardo\v{s}, Kelly, Nelson and Postle provide a family of counterexamples called the \emph{even-plane matroids}. For more information, see [\ref{bkknp}].

Conjecture \ref{chi_bounded_conj} is trivial for $s = 1,2,3$, and was resolved by the authors [\ref{nn2}] for the case $s = 4$. It remains open for $s \geq 5$. 

We can also consider the following more general problem. For $1 \leq s \leq t$, a simple binary matroid, reviewed as a restriction of a finite binary projective geometry $\PG(n-1,2)$, is \emph{$I_{s,t}$-free} if its intersection with any rank-$t$ flat of $\PG(n-1,2)$ is not an $s$-element independent flat.

\begin{conjecture}\label{chi_bounded_conj_ex}
For any $1 \leq s \leq t$, the simple $I_{s,t}$-free and triangle-free binary matroids have bounded critical number. 
\end{conjecture}

It is immediate that Conjecture \ref{chi_bounded_conj} is a special case of Conjecture \ref{chi_bounded_conj_ex} by setting $s=t$. But these two conjectures are in fact equivalent; for any given $s, t$, $1 \leq s \leq t$, there exists a sufficiently large $s'$ for which $I_{s,t}$-freeness implies $I_{s'}$-freeness. 

Our main result is a proof of Conjecture \ref{chi_bounded_conj_ex} for the case $s=1$, making partial progress towards Conjecture \ref{chi_bounded_conj}. 

In this paper, it will be helpful to be explicit about the ambient binary projective geometry in which our matroids are materialised. Therefore, we make the following modifications, and explicitly redefine our terminology correspondingly. A \emph{simple binary matroid} (hereafter just a \emph{matroid}) is a pair $M = (E,G)$ where $G$ is a finite binary projective geometry $\PG(n-1,2)$, and $E$, called the \emph{ground set}, is any subset of the points of $G$. Abusing notation, we write $G$ for the points of $G$. The dimension of $M$ is the dimension $n$ of $G$ as a geometry. Two matroids $(E_1,G_1)$ and $(E_2,G_2)$ are isomorphic if an isomorphism from $G_1$ to $G_2$ maps $E_1$ to $E_2$. Note that for a matroid $M=(E,G)$, we do not require $E$ to span $G$.

The existence of an ambient space allows us to make the following natural definition, in the same way induced subgraphs are defined. A matroid $N$ is an \emph{induced restriction}, or sometimes \emph{induced submatroid}, of $M$ if $N = (E \cap F, F)$ for some subgeometry $F$ of $G$. If $M$ has no induced restriction isomorphic to $N$, then $M$ is \emph{$N$-free}. 

For an $n$-dimensional matroid $M=(E,G)$, the \emph{complement} of $M$ is the matroid $M^c = (G \del E, G)$. We write $\omega(M)$ for the dimension of the largest subgeometry contained in $E$. The \emph{critical number} of $M$ is the quantity $\chi(M) = n - \omega(M^c)$. This parameter $\chi$ is a natural analogue of chromatic number for graphs, although the connection is not obvious at first (see [\ref{oxley}] p. 588 for a discussion). 

A \emph{flat} is a subgeometry of $G$. A \emph{triangle} is a two-dimensional flat, and if $E$ contains no triangle of $G$, then we say that $M$ is \emph{triangle-free}. A \emph{hyperplane} is a maximally proper flat of $G$. If $E$ is a one-element subset of a $t$-dimensional binary projective geometry $G$, we write $I_{1,t}$ for the matroid $(E,G)$. Now, we state our main result.

\begin{theorem}\label{main_theorem}
For any $t \geq 1$, the $I_{1,t}$-free and triangle-free simple binary matroids have bounded critical number.
\end{theorem}

The proof relies on a key lemma, which is derived from Green's regularity lemma for abelian groups ([\ref{g}]). The proof also uses a Ramsey-type result for matroids. Using these two results leads to an extremely large bound on $\chi$, and it will be interesting to what better bounds one can find. 

In addition to making partial progress towards Conjecture \ref{chi_bounded_conj}, we hope that our proof illustrates a potential use of additive combinatorial techniques to solve Conjecture \ref{chi_bounded_conj} and other similar problems. 
\end{section}

\begin{section}{Preliminaries}
In this section, we give a list of results we need to prove our main theorem as black boxes. The main ingredient is the following key lemma, which is an application of Green's regularity lemma for abelian groups [\ref{g}]. We will derive this key lemma in the next section. $\AG(n-1,2)$ is an \emph{affine geometry}, which is a binary projective geometry $\PG(n-1)$ with a hyperplane removed. For any two sets $X,Y$, their sum is $X+Y = \{x+y \mid x \in X, y \in Y\}$. 

\begin{lemma}
For all $0 < \alpha \leq 1$, there exists an integer $l \geq 0$ such that for each $X \subseteq \bF_2^n$ with $|X| \geq \alpha 2^{n}$, $X+X+X$ either contains a subspace of codimension $l$ or affine subspace of codimension $l$. 
\end{lemma}

%For all $0 < \alpha \leq 1$, there exists an integer $l \geq 0$ such that for each $X \subseteq \bF_2^n$ with $|X| \geq \alpha 2^{n}$, $X+X+X$ either contains $\AG(n-l,2)$ or $\PG(n-l-1,2)$. 

In addition, we will use the following well-known results from matroid theory. The first is the matroidal analogue of the multicolour version of Ramsey's Theorem for graphs, and it follows from the Hales-Jewett Theorem [\ref{HJ}]. 

\begin{theorem}\label{Ramsey}
For any $c \geq 1$, $r_1, r_2, \cdots r_c \geq 1$, there exists an integer $N = GR(r_1,r_2,\cdots,r_c)$ such that for all $n \geq N$, if the points of $\PG(n-1,2)$ are coloured with $c$ different colours, it must contain a monochromatic copy of $\PG(r_i -1,2)$ in colour $i$ for some $1 \leq i \leq c$. 
\end{theorem}

In our application, the $r_i$'s will always be the same. We therefore write $GR(c,r)$ for the Ramsey number by only specifying the number of colours $c$ and the dimension $r$ of a monochromatic projective geometry we expect to find. 

We will use the following classical theorem of Bose and Burton, which is the geometric analogue of Tur\'an's theorem [\ref{bb}]. Their theorem also determines the matroids that attain the bound, but we will not need it in this paper.

\begin{theorem}[{[\ref{bb}]}]\label{bbt}
	Let $M = (E,G)$ be a matroid. If $t \ge 0$ and $E$ contains no $(t+1)$-dimensional flat of $G$, then $|E| \le 2^{\dim(M)}(1-2^{-t})$.
\end{theorem} 
\end{section}

\begin{section}{Regularity}
In this section, we will derive our key lemma from Green's regularity lemma for abelian groups [\ref{g}]. 

Let $V=\GF(2)^n$ and let $X \subseteq V$. We say that $X$ is \emph{$\eps$-uniform} if for each hyperplane $H$ of $V$ $$||H \cap X | - |X \del H|| \leq \eps |V|.$$

Let $H$ be a subspace of $V$ and $v \in V$. We call the sets of the form $H+v = \{x+v \mid x \in H\}$ the \emph{cosets} (note that $H$ is considered a coset of $H$ itself). We say that $H$ is \emph{$\eps$-regular} with respect to $V$ and $X$ if for all but an $\eps$-fraction of cosets $H' = H+v$ of $H$, the set $(H' \cap X) + v$ is $\eps$-uniform. Regularity is a measure of how the elements of $X$ are distributed across the different cosets of $H$ in $V$. Green's regularity ensures that there is an $\eps$-regular subspace of bounded codimension. We state the version of his theorem for $V=\GF(2)^n$. 

\begin{lemma}[{[\ref{g}]}]\label{Green's regularity}
For all $0 < \eps < \frac{1}{2}$, there exists $l \in \bN$ such that for each $X \subseteq V$ there is a subspace $H \subseteq V$ of codimension at most $l$ that is $\eps$-regular with respect to $X$ and $V$. 
\end{lemma}

In this lemma, $l$ is guaranteed to be at most $T(\eps^{-3})$, where $T(\alpha)$ is an exponential tower of $2$'s of height $\left \lceil{\alpha}\right \rceil$.

Next we have a counting lemma for $\eps$-uniform sets from [\ref{g}] and [\ref{tv}]. We simply state a special version that is adequate for our proof.

\begin{lemma}[{[\ref{g}], [\ref{tv}]}]\label{counting}
Let $X \subseteq V$ with $|X| = \alpha |V|$. If $0 < \eps < \frac{1}{2}$ and $X$ is $\eps$-uniform, then for any $u \in V$, $$|\{(x_1, x_2, x_3) \in X^3 \mid x_1+x_2+x_3 = u\}| \geq (\alpha^3 - \eps)|V|^2. $$
\end{lemma}

Note that the quantity $\alpha^3 2^{2n}$ is how many solutions to $x_1+x_2+x_3 = u$ one would expect to find if the set $X$ was random. 

We can now prove our key lemma, restated below.

\begin{lemma}\label{keylemma}
For all $0 < \alpha \leq 1$, there exists an integer $l$ such that for each $X \subseteq \GF(2)^n$ with $|X| \geq \alpha 2^{n}$, $X+X+X$ either contains a subspace of codimension $l$ or affine subspace of codimension $l$. . 
\end{lemma}

\begin{proof}
Apply Lemma \ref{Green's regularity} with $\eps = \frac{\alpha^3}{9}$ to obtain an integer $l$, so that we obtain a subspace $H$ of codimension $k \leq l$ that is $\eps$-regular with respect to $X$ and $V$. 

There are exactly $2^k$ cosets of $H$. Pick $A \subseteq V$ so that $H+a$, $a \in A$ are the distinct cosets of $H$. Let
\begin{align*}
	A_{sparse} = \{a \in A : |X \cap (H+a)| < \frac{\alpha}{2}|H| \} 
\end{align*} 
\begin{align*}
	A_{bad} = \{a \in A : (X \cap (H+a))-a \text{ is not $\eps$-uniform} \}
\end{align*} 

Since $H$ is $\eps$-regular with respect to $X$ and $V$, it follows that $|A_{bad}| \leq \eps 2^k$. Note
	\begin{align*}
		\alpha 2^k |H| = \alpha 2^n  \leq |X| &= \sum_{a \in A} |X \cap (H+a)| \\
		 &\leq \frac{\alpha}{2}|H||A_{sparse}| + |H|(|A|-|A_{sparse}|).
	\end{align*} 
	
Therefore, it follows that $\alpha 2^k \leq \frac{\alpha}{2} |A_{sparse}| + 2^k - |A_{sparse}| $. Hence, $|A_{sparse}| \leq 2^k\frac{1-\alpha}{1-\frac{\alpha}{2}} = 2^k(1-\frac{\alpha}{2-\alpha})$. Therefore
	\begin{align*}
		|A| - |A_{sparse}| \geq \frac{\alpha}{2} 2^k > \eps 2^k.
	\end{align*} 

So there exists some $a_0 \in A \del (A_{sparse} \cup A_{bad})$, which means that $|X \cap (H+a_0)| \geq \frac{\alpha}{2} |H|$ and $(X \cap (H+a_0))-a_0$ is $\eps$-uniform. 

Let $u \in H+a_0$. We argue that $u \in X+X+X$. We apply Lemma \ref{counting} with $H$ as the vector space. Then the number of solutions to the equation $x_1+x_2+x_3 = u + a_0$ with $x_1, x_2, x_3 \in (X \cap (H+a_0))-a_0$ is at least
	\begin{align*}
		& \left(\frac{|(X \cap (H+a_0))-a_0|}{|H|}  \right)^3 |H|^2 - \eps|H|^2 \\
		&\geq \left(\frac{\alpha}{2}\right)^3 |H|^2 - \eps |H|^2 \geq \left( \frac{\alpha^3}{8}-\frac{\alpha^3}{9} \right) |H|^2 >0.
	\end{align*} 

So there exists $x_1, x_2, x_3 \in (X \cap (H+a_0))-a_0$ such that $x_1+x_2+x_3 = u+a_0$. Therefore, $(x_1 + a_0) + (x_2 + a_0) + (x_3 + a_0) = u$, as required. If $a_0 \in H$, $X+X+X$ contains a subspace, otherwise $X+X+X$ contains an affine subspace, of codimension $l$. 	
\end{proof}
\end{section}

\begin{section}{The tripods}
We define a class of matroids to which we refer as the \emph{tripods} (not to be confused with tripods in graph theory). Let $T_0 $ be the $1$-dimensional matroid with a one-element ground set. The tripods are constructed recursively as follows. The \emph{$k$-th order tripod} $T_k = (E_k, G_k)$, $k \geq 1$, is the matroid with a codimension-$3$ flat $H$ of $G_k$ and $x,y,z \notin H$ where $\cl(H \cup \{x,y,z\}) = G_k$ so that $E_k = (H \cap E_k) \cup (\{x,y,z\} + (H \cap E_k)) \cup \{x+y+z\}$ and $(E_k \cap H, H) \cong T_{k-1}$. 

The properties of the tripods we will need are as follows. The matroid $C_5$ is the \emph{five-element circuit}, the full-rank matroid whose ground set consists of five points that add to zero. For $t \geq 5$, we write $C_{5,t}$ for the $t$-dimensional matroid $(E,G)$ such that $(E,\cl(E)) \cong C_5$. Note $C_{5,4} \cong C_5$. 

\begin{lemma}\label{tripod_lemma}\
\begin{itemize}
	\item $T_k = (E_k, G_k)$ has dimension $3k+1$,
	\item When $k \geq 1$, there is a flat $F_k$ of dimension $2k+2$ for which $T_k | F_k \cong C_{5,2k+2}$. 
	\item $F_k \subseteq E_k \cup (E_k + E_k)$.
\end{itemize}
\end{lemma}

\begin{proof}
It follows straight from the definition that $T_k$ has dimension $3k + 1$. 

Note that when $k=1$, then $T_1 \cong C_5$, so all the statements are true for $k=1$ as well. From here, we work by induction on $k$. Take $T_k = (E_k, G_k)$ where $k \geq 2$ and assume that the statements are true for smaller values of $k$. 

By definition, we can write $E_k = (H \cap E_k) \cup (\{x,y,z\} + H \cap E_k) \cup \{x+y+z\}$ where $H$ is a codimension-$3$ flat and $x,y,z \notin H$ such that $(E_k \cap H, H) \cong T_{k-1}$ and $\cl(H \cup \{x,y,z\}) = G_k$. By induction, we can find a flat $F_{k-1}$ of dimension $2k$ for which $E_k | F_{k-1} \cong C_{5,2k}$ and $F_{k-1} \subseteq (E_k \cap H) \cup ((E_k \cap H)+ (E_k \cap H))$. We claim that $F_k = \cl(F_{k-1} \cup \{x+y, x+z\})$ suffices. $T_k | F_k \cong C_{5,2k+2}$ is immediate by definition. It remains to show the last statement. Let $v \in F_k \del E_k$. We now perform a case analysis.

If $v \in F_{k-1}$, then $v \in (E_k \cap H) + (E_k \cap H)$ by induction, so $v \in E_k + E_k$ so suppose not. If $v \in {x+y,x+z,y+z}$, then by symmetry we may assume $v=x+y$. Then we may pick any element $t \in E_k \cap H$, and $v = (x+ t) + (y + t) \in E_k + E_k$. The only elements that remain to be checked are of the form $v + w$ where $v \in \{x+y, x+z, y+z\}$ and $w \in F_{k-1}$. By symmetry we may assume that $v = x+y$. If $w \notin F_{k-1} \cap E_k$, then by induction, it follows that there exist two elements $a,b \in H \cap E_k$ for which $w = a + b$. But then $v + w = (a + x) + (b + y) \in E_k + E_k$. Hence, we may assume that $w \in F_{k-1} \cap E_k$. But then $v + w = (x+y+z) + (z+ w) \in E_k + E_k$. This proves finishes the proof.
\end{proof}
\end{section}

\begin{section}{Proof of the Main Theorem}
We are now ready to prove our main theorem. 

\begin{theorem}
For any $t \geq 1$, the $I_{1,t}$-free and triangle-free matroids have bounded critical number.
\end{theorem}

\begin{proof}
The result is trivial when $t = 1$, so suppose $t \geq 2$. We will work by induction on $t$. Suppos the $I_{1,t-1}$-free, triangle-free matroids have critical number at most $c_{t-1}$.

Let $M = (E,G)$ be an $I_{1,t}$-free, triangle-free matroid. We may assume that $E \neq \varnothing$, as otherwise $\chi(M) = 0$, hence $M$ contains $T_0$. Suppose that $M$ contains a $T_k$-restriction (not necessarily induced). Let $G_1$ be the $(3k+1)$-dimensional flat for which $M|G_1$ contains this $T_k$-restriction. 

We claim that $k < \frac{t}{2}$. For a contradiction, suppose not. In particular, $k \geq 1$. Let $F_k \subseteq G_1$ be the $(2k+2)$-dimensional flat obtained from Lemma~\ref{tripod_lemma}, so that $M \rvert F_k$ contains $C_{5,2k+2}$ as a restriction; let $E' \subseteq E \cap F_k$ be the set of five elements corresponding to this $C_5$-restriction. But this is in fact induced; if we take $v \in F_k \del E'$, there exists $v_1,v_2 \in E \cap G_1$ by Lemma~\ref{tripod_lemma} such that $v = v_1+v_2$. Since $M$ is triangle-free, $v \notin E$. Hence $M \rvert F_k \cong C_{5,2k+2}$. Note that $C_{5,2k+2}$ contains an induced $I_{1,2k+1}$-restriction. But we assumed $t \leq 2k$, so this contradicts $I_{1,t}$-freeness. Hence, we may pick the largest integer $k \geq 0$ for which $M$ contains a $T_k$-restriction. 

Now, fix a flat $G_2$ such that $G_1 \cap G_2 = \varnothing$ and $\cl(G_1 \cup G_2) = G$. For each $v \in G_2$, we assign a colour $(e,S)$ where $e=1$ if $v \in E$ and $e=0$ otherwise, and $S = v + ((v+G_1) \cap E)$. Note that $S \subseteq G_1$. Hence we have at most $c = 2^{3k+2}$ colours. 

We now define a new matroid $N=(X,G_2)$, where $X$ consists of elements of $G_2$ for which their colours $(e, S)$ satisfy $(G_1 \cap E) \del S = \varnothing$ and $e=0$. We first claim that $N$ is dense.

\begin{claim}\label{dense}
$|X| \geq 2^{\dim(N)-GR(c,t)}$. 
\end{claim}

\begin{subproof}
If $G_2 \del X$ contains a $GR(c,t)$-dimensional flat, then by Theorem \ref{Ramsey}, we have a monochromatic $t$-dimensional flat $F'$contained in $G_2 \del X$, in colour $(e,S)$. If $e=1$, then $F' \subseteq E$, which contradicts triangle-freeness since $t \geq 2$. So $e=0$, meaning $F' \cap E = \varnothing$. But there also exists $v \in G_1 \cap E$ for which $(F' + v)\cap E = \varnothing$. This means that $M \rvert \cl(\{v\} \cup F')$ is an induced $I_{1,t}$-restriction, a contradiction. 

Now, we apply Theorem \ref{bbt} to conclude that $|G_2 \del X| \leq 2^{\dim(N)}(1-2^{-GR(c,t)+1})$. Hence, $|X| \geq 2^{\dim(N)}2^{-GR(c,t)}$
\end{subproof}

We make one further observation which follows immediately from the definition of tripods.

\begin{claim}\label{thirdpoint}
$(X+X+X) \cap E = \varnothing$. 
\end{claim}

\begin{subproof}
For a contradiction, suppose that there exist $v_1,v_2,v_3 \in X$ for which $v_1+v_2+v_3 \in E$. But then since $G_1$ contains a $k$-th order tripod, $M \rvert (\{v_1,v_2,v_3\} \cup G_1)$ contains a $(k+1)$-th order tripod, contradicting the maximality of $k$. 
\end{subproof}

Let $l$ be the integer obtained from applying Lemma \ref{keylemma} with $\alpha = 2^{-GR(c,t)}$. We now have two outcomes we consider separately.

\textbf{Case 1}: $X + X + X$ contains $\AG(\dim(N)-l,2)$.

We can find a flat $F'$ of codimension $l-1$ in $N$, and a hyperplane $F''$ of $F''$ for which $F' \del F'' \subseteq X+X+X$. In particular, by \ref{thirdpoint}, $(F' \del F'') \cap E = \varnothing$. Hence $M \rvert F''$ is $I_{1,t-1}$-free (and triangle-free). Since $F''$ has codimension at most $l+(3k+1)$ in $M$, we conclude by induction that $\chi(M) \leq c_{t-1} + 3k+l+1$. 

\textbf{Case 2}: $X+X+X$ contains $\PG(\dim(N)-l-1,2)$. 

By \ref{thirdpoint}, it follows that there exists a flat $F'$ of codimension $l$ for which $F' \cap E = \varnothing$. Hence $\chi(M) \leq l+3k+1$. 

Note that $k$ and $l$ are both functions of $t$. This completes the proof by induction on $t$. 
\end{proof}
\end{section}

\section*{References}
\newcounter{refs}
\begin{list}{[\arabic{refs}]}
{\usecounter{refs}\setlength{\leftmargin}{10mm}\setlength{\itemsep}{0mm}}

\item\label{bb}
R. C. Bose, R. C. Burton, 
A characterization of flat spaces in a finite geometry and the uniqueness of the Hamming and the MacDonald codes, 
J. Combin. Theory 1 (1966), 96--104. 

\item\label{bkknp}
M. Bonamy, F. Kardo\v{s}, T. Kelly, P. Nelson, L. Postle,
The structure of binary matroids with no induced claw or Fano plane restriction,
Advances in Combinatorics, 2019:1,17 pp.

\item\label{gn}
J. Geelen, P. Nelson,
The critical number of dense triangle-free binary matroids,
J. Combin. Theory Ser. B 116 (2016), 238-249.

\item\label{gn2}
J. Geelen, P. Nelson,
Odd circuits in dense binary matroids,
Combinatorica 37 (2017), 41-47.

\item\label{gs}
P. Govaerts, L. Storme,
The classification of the smallest non-trivial blocking sets in $\PG(n,2)$,
J. Combin. Theory Ser. A 113 (2006), 1543-1548.

\item\label{g}
B. Green,
A Szemer\'{e}di-type regularity lemma in abelian groups, with applications,
Geometric \& Functional Analysis GAFA 15 (2005), 340-376

\item\label{g85}
A. Gy\'arf\'as, Problems from the world surrounding perfect graphs, Proceedings of the International Conference on Combinatorial Analysis and its Applications, (Pokrzywna, 1985), Zastos. Mat. 19 (1987), 413--441.

\item\label{HJ}
R. A. Hales, R. I. Jewett,
Regularity and positional games,
Trans. Amer. Math. Soc. 106 (1963), 222-229

\item\label{nn1}
P. Nelson, K. Nomoto,
The structure of claw-free binary matroids,\\
arXiv:1807.11543 (2018).

\item\label{nn2}
P. Nelson, K. Nomoto,
The structure of $I_4$-free, triangle-free binary matroids.
arXiv:2005.00089 (2020)

\item\label{nnorin}
P. Nelson, S. Norin,
The smallest matroids with no large independent flat.
arXiv:1909.02045 (2019)

\item \label{oxley}
J. G. Oxley, 
Matroid Theory,
Oxford University Press, New York (2011).

\item\label{s81}
D.P. Sumner, 
Subtrees of a graph and chromatic number, in The Theory and Applications of Graphs, (G. Chartrand, ed.), John Wiley \& Sons, New York (1981), 557--576.

\item \label{tv}
T. C. Tao, V. H. Vu,
Additive Combinatorics, Cambridge Studies in Advanced Mathematics,
105, Cambridge University Press, Cambridge (2006)

\end{list}

\end{document}